\documentclass[12pt,a4paper]{amsart}

\usepackage[l2tabu,orthodox]{nag}

\usepackage{hyperref}

\usepackage{marginnote}
\usepackage{enumerate}

\usepackage[swedish,english]{babel}
\usepackage[utf8]{inputenc}

\usepackage[all]{xy}

\usepackage{a4wide}

\usepackage{amsmath,amssymb,amsthm}

\theoremstyle{plain}
\newtheorem{theorem}{Theorem}[section]
\newtheorem{lemma}[theorem]{Lemma}
\newtheorem{proposition}[theorem]{Proposition}
\newtheorem{cor}[theorem]{Corollary}
\theoremstyle{definition}
\newtheorem{definition}[theorem]{Definition}

\newtheorem{example}[theorem]{Example}

\theoremstyle{remark}
\newtheorem{remark}[theorem]{Remark}

\newcommand{\N}{\mathbb{N}}

\DeclareMathOperator{\Ob}{ob}
\DeclareMathOperator{\Mor}{mor}

\title[Graded von Neumann regularity]{Graded von Neumann regularity \\ of rings graded by semigroups}
\date{\today}

\begin{document}

\author{Daniel L\"{a}nnstr\"{o}m, Johan \"{O}inert}
\address{
Department of Mathematics and Natural Sciences,
Blekinge Institute of Technology,
SE-37179 Karlskrona, Sweden}
\email{daniel.lannstrom@bth.se} \email{johan.oinert@bth.se}


\begin{abstract}
In this article, we give a complete characterization of semigroup graded rings which are graded von Neumann regular.
We also demonstrate our results by applying them to several classes of examples, including matrix rings and groupoid graded rings.
\end{abstract}

\maketitle


\section{Introduction}

In 1936, von Neumann \cite{vonNeumann1936} introduced the class of \emph{regular rings} which nowadays are known as \emph{von Neumann regular rings}.
Recall that a ring $R$ is \emph{von Neumann regular}
if $r\in rRr$ for every $r\in R$.
Note that any division ring, for instance, is von Neumann regular. On the other hand, the ring of integers $\mathbb{Z}$ is not von Neumann regular.
For a thorough account of the theory of von Neumann regular rings, we refer the reader to \cite{GoodearlBook1991}.

Let $S$ be a semigroup. A ring $R$ is said to be \emph{$S$-graded} (or \emph{graded by $S$})
if there is a collection of additive subgroups $\{ R_s \}_{s \in S}$ of $R$ such that $R=\bigoplus_{s \in S} R_s$ and $R_s R_t \subseteq R_{st}$ for all $s,t \in S$. 
If $G$ is a group, then a $G$-graded ring $R$ is said to be \emph{graded von Neumann regular} if, for all $g \in G$ and $r \in R_g$, the relation $r \in rRr$ holds.
This notion was introduced by N\u{a}st\u{a}sescu and van Oystaeyen who also gave a characterization of graded von Neumann regular rings in the class of unital strongly group graded rings (see \cite[Cor.~C.I.5.3]{nastasescu1982}).
In \cite{Lannstrom2020}, the first named author of this article gave a characterization of (non-unital) group graded rings which are graded von Neumann regular.
The purpose of this article is to generalize that result to the setting of rings graded by 
semigroups. 

Recall that a semigroup $S$ is said to be \emph{regular} if $Q(s):=\{x\in S \mid s=sxs \}$ is nonempty for every $s\in S$.
One can show that $Q(s)$ is nonempty for every $s\in S$ if and only if $V(s):=\{x\in S \mid s=sxs \text{ and } x=xsx \}$ is nonempty for every $s\in S$.
The class of regular semigroups contains two subclasses which are of great importance; the inverse semigroups and the completely regular semigroups.
Now, we introduce the notion of a graded von Neumann regular ring in the context of semigroup graded rings (cf. Definition~\ref{def:vNregularGroupoid}).

\begin{definition}\label{def:grVNsemigroup}
Let $S$ be a semigroup.
An $S$-graded ring $R$ is said to be \emph{graded von Neumann regular}
if for all $s\in S$, $r \in R_s$ and $t\in V(s)$,
there is some $y\in R_t$
such that
$r = ryr$.
\end{definition}

Note that if $S$ is a group, then the above definition coincides with the one used by N\u{a}st\u{a}sescu and van Oystaeyen \cite{nastasescu1982}.

Throughout this article, all rings are assumed to be associative, but not necessarily commutative nor unital.
Here is an outline of this article.

In Section~\ref{Sec:NearlyEpsStrongGradings}, we introduce and record characterizations of epsilon-strong, respectively nearly epsilon-strong, semigroup gradings.
In Section~\ref{Sec:MainResult}, we prove the main results of this article (see Theorem~\ref{thm:Main} and Theorem~\ref{thm:InvSemiGroup}).
In Section~\ref{Sec:Examples}, we apply our results to several different classes of examples. Notably, we obtain a complete characterization of groupoid graded rings which are graded von Neumann regular in the sense of Definition~\ref{def:vNregularGroupoid} (see Theorem~\ref{thm:Groupoid}).

\section{Nearly epsilon-strong semigroup gradings}\label{Sec:NearlyEpsStrongGradings}

In this section, we introduce epsilon-strong, respectively nearly epsilon-strong, gradings in the context of semigroup graded rings. We also characterize such gradings (see Propositions~\ref{prop:EpsStrEquiv} and \ref{prop:NearlyEpsEquiv}).
Throughout this section, we let $S$ denote an arbitrary semigroup.

\begin{remark}
Let $R$ be an $S$-graded ring.

(a) For each $e\in E(S)=\{s\in S \mid s^2=s\}$,
$R_e$ is a subring of $R$.

(b) For all $s\in S$ and $t\in V(s)$,
$R_s$ is an $R_{st}$-$R_{ts}$-bimodule.

(c)
For all $s\in S$ and $t\in V(s)$,
$R_s R_{t}$ is an ideal of $R_{st}$,
and $R_t R_s$ is an ideal of $R_{ts}$.
\end{remark}

Recall that a ring $R$ is said to be \emph{left $s$-unital} (resp. \emph{right $s$-unital}) if $x \in Rx$ (resp. $x \in xR$) for every $x\in R$. A ring which is both left $s$-unital and right $s$-unital is said to be \emph{$s$-unital}.
A ring $R$ is said to be \emph{left unital} (resp. \emph{right unital}) if it has a left unity element (resp. a right unity element),
i.e. if there is an element $u \in R$ such that $ur=r$ (resp. $ru=r$) for every $r\in R$.

\emph{Epsilon-strongly} and \emph{nearly epsilon-strongly} group graded rings were introduced and studied in \cite{NOP2018} and \cite{NystedtOinert2020}, respectively. Here we introduce the corresponding notions for semigroup gradings.

\begin{definition}
Let $R$ be an \emph{$S$-graded} ring.
\begin{enumerate}[{\rm (i)}]
	\item If $R_s = R_s R_t R_s$ for all $s\in S$ and $t\in V(s)$, then $R$ is said to be \emph{symmetrically $S$-graded}.

	\item If $R_s R_t = R_{st}$ for all $s,t \in S$, then $R$ is said to be \emph{strongly $S$-graded}.
	
	\item If $R$ is symmetrically $S$-graded, and 
	$R_s R_t$ is unital
	for all $s\in S$ and $t\in V(s)$, then $R$ is said to be \emph{epsilon-strongly $S$-graded}.
	
	\item If $R$ is symmetrically $S$-graded, and 
	$R_s R_t$ is $s$-unital 
	for all $s\in S$ and $t\in V(s)$, then $R$ is said to be \emph{nearly epsilon-strongly $S$-graded}.
\end{enumerate}
\end{definition}

\begin{remark}
When specializing the above definition to the case where $S$ is a group, one recovers the definitions from \cite{NystedtOinert2020,NOP2018}.
\end{remark}

\begin{proposition}\label{prop:EpsStrEquiv}
Let $R$ be an $S$-graded ring.
The following two assertions are equivalent:
\begin{enumerate}[{\rm (i)}]
	\item $R$ is epsilon-strongly $S$-graded;
	\item For all $s\in S$ and $t\in V(s)$, there exist $\epsilon_{s,t} \in R_s R_t$ and $\epsilon_{t,s}' \in R_t R_s$ such that the equalities
$\epsilon_{s,t} r = r = r \epsilon_{t,s}'$ hold for every $r \in R_s$.
\end{enumerate}
\end{proposition}

\begin{proof}
(i)$\Rightarrow$(ii):
Suppose that $R$ is epsilon-strongly $S$-graded.
Take $s\in S$, $t\in V(s)$ and $r\in R_s$.
By assumption, $R_s R_t$ is unital. 
Note that $s \in V(t)$, and hence $R_t R_s$ is also unital.
Let $\epsilon_{s,t}$ and $\epsilon_{t,s}'$ denote the unity elements of $R_s R_t$ and $R_t R_s$, respectively.
Using that $R$ is symmetrically $S$-graded, there exist $n\in \N$, $a_i,c_i \in R_s$ and $b_i \in R_t$, for $i\in \{1,\ldots,n\}$, such that
$r=\sum_{i=1}^n a_i b_i c_i$.
Take $i\in \{1,\ldots,n\}$.
Since $a_i b_i \in R_s R_t$ we get that
$\epsilon_{s,t} a_i b_i = a_i b_i$ and thus
$\epsilon_{s,t} r = \sum_{i=1}^n \epsilon_{s,t} a_i b_i c_i = \sum_{i=1}^n a_i b_i c_i = r$.
Similarly, $r \epsilon_{t,s}' = r$.

(ii)$\Rightarrow$(i):
Suppose that (ii) holds.
Take $s\in S$ and $t\in V(s)$.
It is clear that $R_s R_t R_s \subseteq R_{sts} = R_s$.
Now, we show the reversed inclusion.
Take $r \in R_s$.
By assumption, $\epsilon_{s,t} \in R_s R_t$ and $\epsilon_{s,t} r = r$.
In particular, it follows that
$r = \epsilon_{s,t} r \in R_s R_t R_s$.
This shows that $R$ is symmetrically $S$-graded.
Clearly, $\epsilon_{s,t}$ is a left unity element of $R_s R_t$.
Note that $s\in V(t)$. Hence, there is an element $\epsilon_{s,t}' \in R_s R_t$ such that $r' \epsilon_{s,t}' = r'$ for every $r'\in R_t$.
This means that $\epsilon_{s,t}'$ is a right unity element of $R_s R_t$.
In fact,
$\epsilon_{s,t} = \epsilon_{s,t} \epsilon_{s,t}' = \epsilon_{s,t}'$ is the unity element of $R_s R_t$.
\end{proof}

\begin{cor}
If $R$ is an epsilon-strongly $S$-graded ring, then $R_e$ is unital for every $e\in E(S)$.
\end{cor}

\begin{proof}
Suppose that $R$ is an epsilon-strongly $S$-graded ring.
Take $e\in E(S)$. Note that $e\in V(e)$. There is some $\epsilon_{e,e} \in R_e R_e$ 
and $\epsilon_{e,e}' \in R_e R_e$ such that $\epsilon_{e,e} r = r$ and $r\epsilon_{e,e}' = r$ for every $r\in R_e$.
In particular, $\epsilon_{e,e} = \epsilon_{e,e} \epsilon_{e,e}' = \epsilon_{e,e}'$.
\end{proof}

\begin{proposition}[{Nystedt \cite[Prop.~2.11]{Nystedt2019}}, Tominaga \cite{Tominaga}]\label{Prop:Tominaga}
A ring $R$ is left $s$-unital (resp. right $s$-unital) if and only if for any
finite subset $V$ of $R$ there is $u \in R$ such that for every $v \in V$ the equality $uv = v$ (resp. $vu=v$) holds.
\end{proposition}

\begin{proposition}\label{prop:NearlyEpsEquiv}
Let $R$ be an $S$-graded ring.
The following two assertions are equivalent:
\begin{enumerate}[{\rm (i)}]
	\item $R$ is nearly epsilon-strongly $S$-graded;
	\item For all $s\in S$, $t\in V(s)$ and $r\in R_s$, there exist $\epsilon_{s,t}(r) \in R_s R_t$ and $\epsilon_{t,s}'(r) \in R_t R_s$ such that  the equalities
$\epsilon_{s,t}(r) r = r = r \epsilon_{t,s}'(r)$ hold.
\end{enumerate}
\end{proposition}

\begin{proof}
(i)$\Rightarrow$(ii):
Suppose that $R$ is nearly epsilon-strongly $S$-graded. 
Take $s\in S$, $t\in V(s)$ and $r\in R_s$.
By assumption, $R_s R_t$ is $s$-unital.
Note that $s\in V(t)$, and hence $R_t R_s$ is also $s$-unital.
Using that $R$ is symmetrically $S$-graded, 
there exist $n\in \N$, $a_i,c_i \in R_s$ and $b_i \in R_t$, for $i\in \{1,\ldots,n\}$, such that
$r=\sum_{i=1}^n a_i b_i c_i$.
Since $a_i b_i \in R_s R_t$ and $b_i c_i \in R_t R_s$ for every $i\in \{1,\ldots,n\}$, 
from 
Proposition~\ref{Prop:Tominaga},
it follows  that there exist $\epsilon_{s,t}(r) \in R_s R_t$ and $\epsilon_{t,s}'(r) \in R_t R_s$ such that the equalities $\epsilon_{s,t}(r) a_i b_i = a_i b_i$ and $b_i c_i \epsilon_{t,s}'(r) = b_i c_i$ hold for every $i\in \{1,\ldots,n\}$.
Thus, $\epsilon_{s,t}(r) r = r = r \epsilon_{t,s}'(r)$.

(ii)$\Rightarrow$(i):
Suppose that (ii) holds.
Symmetry of the $S$-grading is shown in the same way as for Proposition~\ref{prop:EpsStrEquiv}.
Take $s\in S$, $t\in V(s)$ and $r\in R_s$.
By assumption, there is $\epsilon_{s,t}(r) \in R_s R_t$ such that $\epsilon_{s,t}(r) r = r$.
Thus, $R_s R_t$ is left $s$-unital.
Note that $s \in V(t)$. Take $r'\in R_t$.
By assumption, there is $\epsilon_{s,t}'(r') \in R_s R_t$ such that $r' \epsilon_{s,t}'(r') = r'$.
Thus, $R_s R_t$ is right $s$-unital.
We conclude that $R_s R_t$ is $s$-unital.
\end{proof}

\section{Graded von Neumann regularity}\label{Sec:MainResult}

In this section, we establish the main results of this article (see Theorem~\ref{thm:Main} and Theorem~\ref{thm:InvSemiGroup}).
Throughout this section, unless stated otherwise, $S$ is assumed to be an arbitrary semigroup.
Recall that $E(S)$ denotes the set of idempotents in $S$.

\begin{lemma}\label{lem:necessary}
If $R$ is an $S$-graded ring which is graded von Neumann regular, then the following two assertions hold:
\begin{enumerate}[{\rm (i)}]
	\item $R_e$ is von Neumann regular for every $e\in E(S)$;
	\item $R$ is nearly epsilon-strongly $S$-graded.
\end{enumerate}
\end{lemma}

\begin{proof}
Suppose that $R$ is $S$-graded and graded von Neumann regular.

(i)
Take $e\in E(S)$ and $r\in R_e$.
Using that $e=e^3$, by graded von Neumann regularity,
there is some $y \in R_e$ such that
$r=ryr$. Thus, $R_e$ is von Neumann regular.

(ii)
Take $s\in S$, $t\in V(s)$ and $r\in R_s$.
Using that $S$ is graded von Neumann regular, there is some $y \in R_t$
such that $r=ryr$. Put $\epsilon_{s,t}(r):= ry \in R_s R_t$ and $\epsilon_{t,s}'(r) := yr \in R_t R_s$.
Clearly,
$\epsilon_{s,t}(r) r = ryr =r$
and $r \epsilon_{t,s}'(r) = ryr = r$.
By Proposition~\ref{prop:NearlyEpsEquiv} we conclude that $R$ is nearly epsilon-strongly $S$-graded.
\end{proof}

We recall the following well-known characterization of von Neumann regularity. For a proof, see e.g. \cite[Proposition 2.2]{Lannstrom2020}.

\begin{proposition}\label{prop:vNreg}
Let $T$ be an $s$-unital ring. The following three assertions are equivalent:
\begin{enumerate}[{\rm (i)}]
    \item $T$ is von Neumann regular;
    \item Every principal left (resp. right) ideal of $T$ is generated by an idempotent;
    \item Every finitely generated left (resp. right) ideal of $T$ is generated by an idempotent.
\end{enumerate}
\end{proposition}

\begin{lemma}\label{lemma:technical}
Let $R$ be a nearly epsilon-strongly $S$-graded ring and suppose that $R_e$ is von Neumann regular for every $e\in E(S)$.
Then, for all $s\in S$, $t\in V(s)$ and $r\in R_s$, the left $R_{ts}$-ideal $R_t r$ is generated by an idempotent in $R_{ts}$.
\end{lemma}

\begin{proof}
Take $s\in S$, $t\in V(s)$ and $r\in R_s$.
Put $I := R_t r$.
Immediately from the grading, it follows that $R_t r \subseteq R_t R_s \subseteq R_{ts}$,
and that $R_{ts} R_t r \subseteq R_{tst} r = R_t r$.
Thus, $R_t r$ is a left ideal of $R_{ts}$.

Using that $R$ is nearly epsilon-strongly $S$-graded,
there is an element $y \in R_s R_t$ such that $yr=r$.
We may write $y = \sum_{i=1}^k a_i b_i$ where $a_1, \ldots, a_k \in R_s$ and $b_1, \ldots, b_k \in R_t$.
Note that $c_i := b_i r \in R_t R_s \subseteq R_{ts}$ with $ts \in E(S)$.
Let $d \in R_t r$ be arbitrary. We may find some $r' \in R_t$ such that $d = r'r$.
Note that $r' a_i \in R_t R_s \subseteq R_{ts}$
and
$d = r'r = r' yr = \sum_{i=1}^k r' a_i b_i r = \sum_{i=1}^k (r' a_i) (b_i r) = \sum_{i=1}^k (r' a_i) c_i$.
This shows that $I = \sum_{i=1}^k R_{ts} c_i$.
Thus, $I$ is finitely generated.
Using that $ts \in E(S)$, $R_{ts}$ is von Neumann regular and $s$-unital by assumption.
By Proposition~\ref{prop:vNreg}, $I$ is generated by an idempotent in $R_{ts}$.
\end{proof}

\begin{theorem}\label{thm:Main}
Let $S$ be a semigroup and let $R$ be an $S$-graded ring. 
The following two assertions are equivalent:
\begin{enumerate}[{\rm (i)}]
	\item $R$ is graded von Neumann regular;
	\item $R$ is nearly epsilon-strongly $S$-graded, and $R_e$ is von Neumann regular for every $e\in E(S)$. 
\end{enumerate}
\end{theorem}

\begin{proof}
(i)$\Rightarrow$(ii):
This follows immediately from Lemma~\ref{lem:necessary}.

(ii)$\Rightarrow$(i):
Suppose that (ii) holds.
Take $s\in S$, $t\in V(s)$ and $r \in R_s$.
We need to show that there is some $y \in R_t$, such that $r=ryr$.
By Lemma~\ref{lemma:technical} there is an idempotent $u \in R_{ts}$ such that $R_{t} r = R_{ts} u$.
Note that $u=u^2 \in R_{ts} u = R_t r$.
Thus, we may find some $r' \in R_t$ such that $u=r'r$.
Moreover,
\begin{displaymath}
	R_s R_t r = R_s (R_t r) = R_s (R_{ts} u) \subseteq R_{sts} u = R_s u.
\end{displaymath}
Using that $R$ is nearly epsilon-strongly $S$-graded, there is some $\epsilon_{s,t}(r) \in R_s R_t$
such that $\epsilon_{s,t}(r) r = r$.
Thus, from the above equation we get that
$r = \epsilon_{s,t}(r) r \in R_s R_t r \subseteq R_s R_{ts} u \subseteq R_{sts} u = R_s u$.
Hence, there is some $r'' \in R_s$ such that $r=r'' u$.
We get that
$ru = (r''u)u = r'' (u^2) = r'' u = r$.
Thus,
$r = ru = r (r' r)$.
This shows that $R$ is graded von Neumann regular.
\end{proof}

\begin{cor}\label{Cor:AssumeEpsilonStrong}
Let $S$ be a semigroup and let $R$
be an epsilon-strongly $S$-graded ring.
Then $R$ is graded von Neumann regular if and only if
$R_e$ is von Neumann regular for every $e\in E(S)$.
\end{cor}

\begin{remark}
Let $R$ be a strongly $S$-graded ring.
Note that $R$ is nearly epsilon-strongly $S$-graded if and only if $R_e$ is $s$-unital for every $e\in E(S)$.
\end{remark}

\begin{cor}\label{cor:strong}
Let $S$ be a semigroup and let $R$
be a strongly $S$-graded ring for which $R_e$ is $s$-unital for every $e\in E(S)$.
Then $R$ is graded von Neumann regular if and only if
$R_e$ is von Neumann regular for every $e\in E(S)$.
\end{cor}

Recall that a semigroup $S$ is said to be an \emph{inverse semigroup} if, for every $s\in S$, the set $V(s)$ contains exactly one element.

\begin{theorem}\label{thm:InvSemiGroup}
Let $S$ be an inverse semigroup and let $R$ be an $S$-graded ring.
The following three assertions are equivalent:
\begin{enumerate}[{\rm (i)}]
	\item $R$ is graded von Neumann regular;
	\item For all $s\in S$ and $r \in R_s$,
there exist $t \in V(s)$ and $y\in R_t$
such that
$r = ryr$;
	\item $R$ is nearly epsilon-strongly $S$-graded, and $R_e$ is von Neumann regular for every $e\in E(S)$. 
\end{enumerate}
\end{theorem}

\begin{proof}
(i)$\Leftrightarrow$(iii): This follows from Theorem~\ref{thm:Main}.

(i)$\Leftrightarrow$(ii): This follows since $\lvert V(s) \lvert =1$ for every $s\in S$. 
\end{proof}

\section{Applications}\label{Sec:Examples}

In this section we present several classes of examples to which our main results can be applied.

\subsection{Semigroup rings}

Let $A$ be an $s$-unital ring and let $S$ be a semigroup. The \emph{semigroup ring} $A[S]$ consists of all finite sums of the form $\sum_{s \in S} a_s \delta_s$, where $a_s \in A$ and $\delta_s$ is a formal symbol for every $s\in S$.
Addition is defined in the natural way and multiplication is given by extending the rule
$a \delta_s a' \delta_{s'} := aa' \delta_{ss'}$ for $a,a'\in A$ and $s,s'\in S$.
Note that $A[S]$ is strongly $S$-graded with the canonical grading defined by $(A[S])_s := A\delta_s$, for $s\in S$.
By Corollary~\ref{cor:strong}, we get the following result.

\begin{proposition}
Let $A$ be an $s$-unital ring and let $S$ be a semigroup containing at least one idempotent.
Equip the semigroup ring $A[S]$ with the canonical $S$-grading.
Then $A[S]$ is graded von Neumann regular if and only if $A$ is von Neumann regular.
\end{proposition}

\subsection{Matrix rings}

\begin{example}[{cf.~\cite[Sect.~3.9]{Kelarev2002}}]
Let $A$ be a unital ring and consider the full matrix ring $M_n(A)$ for some arbitrary $n > 0$. For $i, j \in \{1, \dots, n \}$, let $e_{i,j}$ denote the standard matrix unit, i.e. the matrix with $1_A$ in position $(i,j)$ and zeros everywhere else. Moreover, note that $B_n := \{ 0 \} \cup \{ e_{i,j} \mid 1 \leq i \leq n, 1 \leq j \leq n \}$ is a semigroup under matrix multiplication. By putting $(M_n(A))_0 := \{0\}$ and $(M_n(A))_s := As$ for $s \in B_n \setminus \{0\}$, we get a $B_n$-grading on $M_n(A)$.
\end{example}

\begin{example}
Let $A$ be a unital ring, put $n=3$ and consider $M_3(A)$ with its $B_3$-grading defined in the previous example. Note that $B_3 = \{ 0, e_{1,1}, e_{2,2}, e_{3,3}, e_{1,2}, e_{1,3}, e_{2,1}, e_{2,3}, e_{3,1}, e_{3,2} \}.$
We note that the above $B_3$-grading on $M_3(A)$ is epsilon-strong.
Using Corollary~\ref{Cor:AssumeEpsilonStrong},
we conclude that
$M_3(A)$ is graded von Neumann regular (with respect to the $B_3$-grading) if and only if $A$ is von Neumann regular.
\end{example}

Following \cite{nastasescu2004methods}, we shall say that a semigroup grading on a matrix ring $M_n(A)$ is \emph{good}
if every standard matrix unit $e_{i,j}$ is homogeneous, i.e. if for all $i,j \in \{1,\ldots,n\}$ there is some $s\in S$
such that $e_{i,j} \in (M_n(A))_s$.

\begin{lemma}\label{lem:oppositedegree}
Let $A$ be a unital ring, let $n$ be a positive integer and let $S$ be an inverse semigroup defining a good grading on $M_n(A)$.
If $\text{deg}(e_{i,j})=s$, then $\text{deg}(e_{j,i})=s^{-1}$.
\end{lemma}

\begin{proof}
Suppose that $\text{deg}(e_{i,j})=s$. Using that the grading is good, there is some $t\in S$ such that $\text{deg}(e_{j,i})=t$.
By inspecting the equations $e_{i,j} e_{j,i} e_{i,j} = e_{i,j}$ and $e_{j,i} = e_{j,i} e_{i,j} e_{j,i}$
we conclude that $sts=s$ and $t=tst$, i.e. $t=s^{-1}$.
\end{proof}

\begin{proposition}
Let $A$ be a unital ring, let $n$ be a positive integer and let $S$ be an inverse semigroup defining a good grading on $M_n(A)$.
Furthermore, suppose that for every $e \in E(S)$,
we have either $(M_n(A))_e = \{0\}$
or that there are
 integers $i_1,i_2,\ldots, i_k \in \{1,\ldots,n\}$
 such that $(M_n(A))_e = A e_{i_1,i_1} + A e_{i_2,i_2} + \ldots + A e_{i_k,i_k}$.
Then $M_n(A)$ is graded von Neumann regular if and only if $A$ is von Neumann regular.
\end{proposition}

\begin{proof}
We begin by showing that $R:=M_n(A)$ is epsilon-strongly $S$-graded
using Proposition~\ref{prop:EpsStrEquiv}.
Take $s\in S$. If $R_s = \{0\}$, then there is nothing to check. Otherwise, using that the grading is good, we may write
$R_s = A e_{i_1,j_1} + A e_{i_2,j_2} + \ldots + A e_{i_k,j_k}$ for some $i_1,\ldots,i_k,j_1,\ldots,j_k$.
Supported by Lemma~\ref{lem:oppositedegree},
we may put
$\epsilon_{s,s^{-1}} := e_{i_1,i_1} + e_{i_2,i_2} + \ldots + e_{i_k,i_k} \in R_s R_{s^{-1}}$ and note that $\epsilon_{s,s^{-1}} r = r$ for every $r\in R_s$.
One can show that $R_s R_{s^{-1}} R_s = R_s$
in the same way as in the proof of Proposition~\ref{prop:EpsStrEquiv}.
Also note that $r'\epsilon_{s,s^{-1}}=r'$ for every $r' \in R_{s^{-1}}$.

The ''only if'' statement follows from Corollary~\ref{Cor:AssumeEpsilonStrong}
and the fact that $A e_{1,1}$ is of idempotent degree.
Indeed, if we suppose that $\text{deg}(e_{1,1})=s$ and inspect the equation
$e_{1,1}e_{1,1} = e_{1,1}$, then we conclude that 
$R_{s^2} \cap R_s \neq \{0\}$, which implies $s^2=s$.

Now, we show the ''if'' statement. Suppose that $A$ is von Neumann regular. Take $e\in E(S)$. By assumption, $R_e=\{0\}$ or $R_e = A e_{i_1,i_1} + A e_{i_2,i_2} + \ldots + A e_{i_k,i_k}$ and it is readily verified that $R_e$ is von Neumann regular.
The desired conclusion now follows from Corollary~\ref{Cor:AssumeEpsilonStrong}.
\end{proof}

\subsection{Groupoid graded rings}

Throughout this section, we let $G$ denote an arbitrary groupoid, i.e. a small category in which every morphism is invertible.
The family of objects and morphisms of $G$ will be denoted by $\Ob(G)$ and $\Mor(G)$, respectively. As usual one identifies an object $e$ with the identity morphism $\text{Id}_e$, and hence $\Ob(G) \subseteq \Mor(G)$. 
To ease notation, we will write $G$ instead of $\text{mor}(G)$.
If $g \in G$, then the domain and codomain of $g$ will be denoted by $d(g)$ and $c(g)$, respectively. We let $G^{(2)}$ denote the set of all pairs $(g, h) \in G \times G$ that are composable, i.e. such that $d(g)=c(h)$.

Recall that a ring $R$ is said to be \emph{$G$-graded}
if there are additive subgroups $\{R_g\}_{g\in G}$ of $R$
such that
$R = \oplus_{g\in G} R_g$
and, for $g,h\in G$, we have
$R_g R_h \subseteq R_{gh}$ if $(g,h)\in G^{(2)}$,
and $R_g R_h = \{0\}$ otherwise.

\begin{definition}
Let $R$ be a \emph{$G$-graded} ring.
\begin{enumerate}[{\rm (i)}]
	\item If $R_g = R_g R_{g^{-1}} R_g$ for every $g\in G$, then $R$ is said to be \emph{symmetrically $G$-graded}.
	
	\item If $R$ is symmetrically $G$-graded, and 
	$R_g R_{g^{-1}}$ is unital
	for every $g\in G$, then $R$ is said to be \emph{epsilon-strongly $G$-graded}.
	
	\item If $R$ is symmetrically $G$-graded, and 
	$R_g R_{g^{-1}}$ is $s$-unital 
	for every $g\in G$, then $R$ is said to be \emph{nearly epsilon-strongly $G$-graded}.
\end{enumerate}
\end{definition}

\begin{remark}
It is not difficult to see that (ii) above
is equivalent (cf. \cite[Proposition 37]{NOP2020})
to the definition of an epsilon-strongly groupoid graded ring introduced in \cite[Definition 34]{NOP2020}.
\end{remark}

For groupoid graded rings it is suitable to make the following definition (cf. \cite{nastasescu1982} and Definition~\ref{def:grVNsemigroup}).

\begin{definition}\label{def:vNregularGroupoid}
Let $G$ be a groupoid.
A $G$-graded ring $R$ is said to be \emph{graded von Neumann regular}
if for all $g\in G$ and $r \in R_g$,
the relation $r \in rRr$ holds.
\end{definition}

\begin{remark}\label{rem:grVNregGroupoid}
Note that a $G$-graded ring $R$ is graded von Neumann regular
if and only if for all $g\in G$ and $r\in R_g$,
there is some $y\in R_{g^{-1}}$ such that $r=ryr$.
\end{remark}

The proof of the following result follows that of Proposition~\ref{prop:NearlyEpsEquiv} closely and is therefore omitted.

\begin{proposition}\label{prop:NearlyEpsEquivGroupoid}
Let $G$ be a groupoid and let $R$ be a $G$-graded ring.
The following two assertions are equivalent:
\begin{enumerate}[{\rm (i)}]
	\item $R$ is nearly epsilon-strongly $G$-graded;
	\item For all $g\in G$ and $r\in R_g$, there exist $\epsilon_{g,g^{-1}}(r) \in R_g R_{g^{-1}}$ and $\epsilon_{g^{-1},g}'(r) \in R_{g^{-1}} R_g$ such that the equalities
$\epsilon_{g,g^{-1}}(r) r = r = r \epsilon_{g^{-1},g}'(r)$ hold.
\end{enumerate}
\end{proposition}

We will now associate a certain inverse semigroup $S(G)$ with the groupoid $G$.
Indeed, put $S(G) := G \cup \{0_S\}$
and define a binary operation $*_{S(G)}$ on $S(G)$ as follows.
For all $g,h \in G$ we define:
\begin{itemize}
	\item $g*_{S(G)} h := gh$ if $(g,h) \in G^{(2)}$,
	\item $g*_{S(G)} h := 0_S$ if $(g,h) \notin G^{(2)}$,
	\item $g*_{S(G)} 0_S := 0_S$,
$0_S *_{S(G)} h:= 0_S$, and
$0_S *_{S(G)} 0_S:= 0_S$.
\end{itemize}

It is readily verified that $S(G)$ is an inverse semigroup.
If $R$ is a $G$-graded ring, then we may endow $R$ with an $S(G)$-grading by keeping 
$R_s$ for every $s\in G$ and simply adding the component $R_{0_S} := \{0\}$.

The proof of the following result is straightforward.

\begin{proposition}\label{prop:SwitchGradings}
Let $G$ be a groupoid and let $R$ be a $G$-graded ring.
Simultaneously consider $R$ as an $S(G)$-graded ring, where $S(G)$ is the semigroup associated with $G$.
Then $R$ is epsilon-strongly $G$-graded (resp. nearly epsilon-strongly $G$-graded)
if and only if $R$ is epsilon-strongly $S(G)$-graded (resp. nearly epsilon-strongly $S(G)$-graded).
\end{proposition}

The following result generalizes
\cite[Theorem~1.2]{Lannstrom2020}
and
\cite[Proposition~3.5]{BagioPaquesPinedo}.

\begin{theorem}\label{thm:Groupoid}
Let $G$ be a groupoid and let $R$ be a $G$-graded ring.
The following three assertions are equivalent:
\begin{enumerate}[{\rm (i)}]
	\item $R$ is graded von Neumann regular;
	\item For all $g\in G$ and $r \in R_g$,
there exists $y\in R_{g^{-1}}$
such that
$r = ryr$;
	\item $R$ is nearly epsilon-strongly $G$-graded, and $R_{e}$ is von Neumann regular for every $e\in \Ob(G)$. 
\end{enumerate}
\end{theorem}

\begin{proof}
The proof follows by combining
Remark~\ref{rem:grVNregGroupoid},
Proposition~\ref{prop:SwitchGradings}
and
Theorem~\ref{thm:InvSemiGroup}.
\end{proof}

\end{document}